\documentclass[12pt]{amsart}
\usepackage{amssymb, latexsym, amsmath, amsthm, verbatim}
\usepackage{fullpage}
\usepackage[pdftex]{graphicx}
\newtheorem{theorem}{Theorem}
\newtheorem{lemma}[theorem]{Lemma}

\newtheorem{proposition}[theorem]{Proposition}

\newcommand{\R}{\mathbf{R}}
\renewcommand{\S}{\mathbf{S}}

\title{Non-proper helicoid-like limits of closed minimal surfaces in
  $3$-manifolds}

\author{Maria Calle}
\address{Max Planck Institute for Gravitational Physics\\ Albert Einstein Institute\\ Am M\"{u}hlenberg 1\\ D-14467 Golm\\ Germany.}
\email{maria.calle@aei.mpg.de}
 \author{Darren Lee}
\address{Courant Institute of Mathematical
  Sciences\\ 251 Mercer Street\\ New York, NY
  10012.}
\email{leed@cims.nyu.edu}

\begin{document}

\begin{abstract}
We show that there exists a metric with positive scalar curvature
on $\mathbf{S}^2\times\mathbf{S}^1$ and a sequence of embedded
minimal cylinders that converges to a minimal lamination that, in
a neighborhood of a strictly stable $2$-sphere, is smooth except
at two helicoid-like singularities on the $2$-sphere.  The
construction is inspired by a recent example by D. Hoffman and B.
White.
\end{abstract}

\maketitle

\section{Introduction}

Roughly speaking, it is expected that the only two types of singular
laminations that can occur as limits of sequences of closed embedded minimal
surfaces in a $3$-manifold with positive scalar curvature are
accumulations of catenoids
and non-proper helicoid-like limits.  Recall that a lamination is a
foliation that does not necessarily fill the entire space; in particular,
just like in a foliation, the leaves of a lamination must be locally
parallel graphs.  An example of the first type of limit was constructed by
Colding and De Lellis in \cite{colding-delellis}.  Prior to the
construction given here, no non-proper helicoid-like limits were known to
exist as limits of closed surfaces.  We construct such a limit in
$\S^2\times\S^1$ where two helicoid-like singularities lie on a strictly
stable $2$-sphere.

For closed Riemannian manifolds with positive Ricci curvature, combining the
work of Choi-Wang \cite{choi-wang} and Yang-Yau \cite{yang-yau}
gives an area bound for embedded minimal surfaces that depends only on the
lower bound for the Ricci curvature of the manifold and on the genus of the
surface.  In particular,
constructions such as those in \cite{colding-delellis} and in this paper
are not possible in manifolds of positive Ricci curvature as both of these
constructions require sequences of closed minimal surfaces with unbounded
area.  See also \cite{choi-schoen}, where Choi and Schoen proved that any
limit of closed minimal surfaces in a closed $3$-manifold with positive
Ricci curvature must be smooth.

An immersed surface $\Sigma\subset M$ is said to be proper if the
intersection of $\Sigma$ with any compact subset of $M$ is compact.
Similarly, a lamination is proper if each leaf is proper.  The study of
the properness of minimal surfaces in Euclidean $3$-space has a long
history; see in particular the work on the Calabi-Yau conjectures
\cite{alarcon-ferrer-martin}, \cite{colding-minicozzi5},
\cite{collin-rosenberg}, \cite{jorge-xavier}, \cite{martin-morales}, and
\cite{nadirashvili}.  For limit laminations in Euclidean $3$-space, many
results about properness are known.  In particular, Colding and Minicozzi
showed in \cite{colding-minicozzi6}, \cite{colding-minicozzi7},
\cite{colding-minicozzi8}, and \cite{colding-minicozzi3} that any sequence
of embedded minimal disks in balls of increasing, unbounded radius in
$\mathbf{R}^3$ must converge off of a curve to a foliation of planes, and
so non-proper limits are impossible.  However, in
\cite{colding-minicozzi2}, Colding and Minicozzi construct a sequence of
embedded minimal disks in a fixed ball in $\mathbf{R}^3$ with boundaries
in the boundary of the ball that converges to a limit which is not proper;
see also \cite{meeks} and \cite{meeks2}.

Our construction is inspired by a beautiful variational construction by D.
Hoffman and B. White of the genus one helicoid \cite{hoffman-white};
see also the earlier known constructions of the genus one helicoid
\cite{hoffman-weber-wolf1} and \cite{hoffman-weber-wolf2}.

We consider the case of a manifold which is topologically
$\mathbf{S}^2\times\mathbf{S}^1$, for which we have the following theorem:

\begin{theorem}\label{thm1}
There exists a metric with positive scalar curvature on
$M=\mathbf{S}^2\times\mathbf{S}^1$ and a sequence of embedded
minimal cylinders $\{\Sigma_n\}$ with boundary in an unstable
$2$-sphere that, in a neighborhood of a strictly stable $2$-sphere
$\Gamma$, converges to a minimal lamination that is smooth except
at two helicoid-like singularities on $\Gamma$.

Specifically, in a neighborhood $\Omega$ of the strictly stable
$2$-sphere, $\Sigma_n\cap\Omega\backslash\Gamma$ converges to two
non-properly embedded minimal surfaces $\Sigma^{\pm}$, one on each
side of $\Gamma$, with
$\overline{\Sigma^{\pm}}\backslash\Sigma^\pm = \Gamma$.
 Furthermore, away from two axes, $\Sigma^\pm$ is the union of two
embedded minimal disks $\Sigma^{\{A,B\}\pm}$, each of which spirals
into $\Gamma$.  (See Figure~\ref{picG}.)
\end{theorem}

\begin{figure}[ht]
\begin{center}
\includegraphics{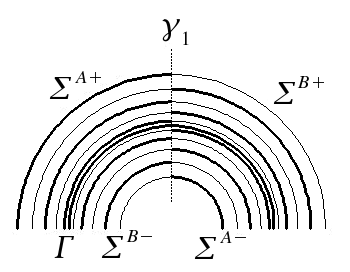}
\caption{\small\sffamily A schematic cross-section near one of the
helicoid axes of the limit lamination in Theorem~\ref{thm1}.  The
limit consists of the strictly stable $2$-sphere $\Gamma$ and four
embedded minimal disks joined along the central axis $\gamma_1$.
$\Sigma^{A+}$ and $\Sigma^{B+}$ lie on one side of $\Gamma$ and
$\Sigma^{A-}$ and $\Sigma^{B-}$ lie on the other side.  Each of
the four spirals infinitely into $\Gamma$.  These surfaces connect
to a symmetric configuration around a second axis $\gamma_2$.}
\label{picG}
\end{center}
\end{figure}

The metric in Theorem~\ref{thm1} will be a warped product metric
with particular assumptions about the warping factor.  If we
instead use the product metric on $\S^2\times\S^1$ (where each
$\S^2\times\{z\}$ is minimal but not strictly stable), the same
construction will give us a sequence of embedded minimal cylinders that
converges smoothly away from two circles to the {\it smooth}
foliation of $\S^2\times\S^1$ by the parallel minimal spheres
$\S^2\times\{z\}$.  Each surface in this sequence looks like it is
obtained by gluing together two oppositely oriented helicoids.
 Each helicoid has a circular axis $\{x\}\times \S^1$ and
$\{-x\}\times \S^1$, which are orthogonal to each $2$-sphere
$\S^2\times \{z\}$ and are ``antipodal'' (they go through $x$ and
$-x$, respectively).

\begin{figure}[ht]
 \begin{center}
 \includegraphics[width=2.5in]{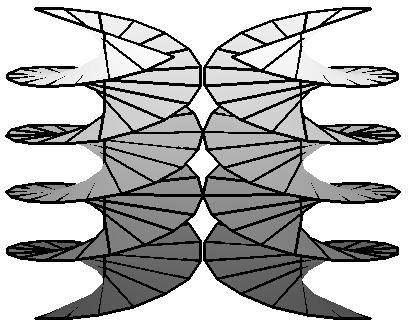}
 \includegraphics[width=2in]{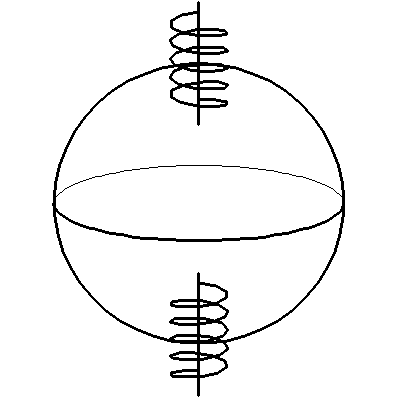}
 \caption[Two helicoids spiraling into a sphere]{\small\sffamily On the left, two helicoids with opposite
 orientation glued together.  On the right, two helicoids spiral around the
 polar axes of $\mathbf{S}^2$.}\label{picAB}
\end{center}
\end{figure}

In section \ref{sect2} we will construct the sequence of minimal
surfaces $\Sigma_n$.  The surfaces in our sequence will be
embedded minimal cylinders that spiral an increasing number of
times around two polar axes $\gamma_1$ and $\gamma_2$ of the form
$\{x\}\times \S^1$.  They spiral around the axes in opposite
directions with the corresponding layers connected.  (See Figure
\ref{picAB}).  In order to find these surfaces, we wish to solve
Plateau problems in the universal cover $\Omega'$ of
$\Omega\backslash(\gamma_1\cup\gamma_2)$, where $\Omega$ is a
neighborhood of a strictly stable $2$-sphere $\Gamma$.

After constructing the sequence, in section \ref{sect3} we will
take the limit when $n\rightarrow\infty$ and obtain a minimal
lamination.  By choosing a sufficiently small neighborhood of the
strictly stable $2$-sphere, we can assume that $\Gamma$ is the
only closed minimal surface in $\Omega$.  Thus, this lamination
must consist of the strictly stable $2$-sphere and two minimal
surfaces, one on each side, that spiral into it.  In particular,
these two minimal surfaces will be embedded but not proper.

We would like to thank our adviser Tobias Colding for his many
helpful suggestions.  We are also thankful to the referee for the many comments that have greatly improved the exposition of this paper.

\section{Preliminaries}\label{sect1}

We parametrize $M=\mathbf{S}^2\times\mathbf{S}^1$ by $(\phi,\theta,z)\in[0,\pi]\times[0,2\pi)\times(-\pi,\pi]$, where $(\phi,\theta)$ are latitudinal and longitudinal coordinates on
$\mathbf{S}^2$ and $z\in (-\pi,\pi]\sim\mathbf{S}^1$.  We consider a warped product metric in $M$ given by $$g^2 = \omega^2(z)(d\phi^2 +
\sin^2\phi\; d\theta^2)+dz^2.$$ We will
assume that $\omega$ is a smooth function $\mathbf{S}^1\to\R$ such
that:

\begin{enumerate}
 \item $\omega(\pi)>\omega(0)=1$,
 \item $\omega(z) = \omega(-z)$,
 \item $\omega''(0)>0$, and
 \item $z=0,\pi$ are the only
critical points of $\omega$.
\end{enumerate}

With these assumptions, $\Gamma=\{z=0\}$ is a strictly stable
minimal $2$-sphere and $\{z=\pi\}$ is an unstable minimal
$2$-sphere.  We will be working in the neighborhood
$\Omega=M\backslash\{z=\pi\}$ of the strictly stable $2$-sphere $\Gamma$.
 Note that the boundary $\partial\Omega=\S^2\times\{\pi,-\pi\}$ is
minimal and hence weakly mean convex.  Recall that $\partial\Omega$ is weakly mean convex if its mean curvature is $H\ge 0$ with respect to the unit normal vector pointing towards $\Omega$.  Also, $\Gamma$ is the only
closed minimal surface in $\Omega$:

\begin{lemma}\label{no-closed-minimal-surfaces}
 If $\Sigma$ is a closed minimal surface in $M$ away from $\{z=\pi\}$, then $\Sigma=\Gamma$.
\end{lemma}

In order to prove this lemma, we will need the following maximum
principle of Solomon and White (see \cite{solomon-white}):

\begin{theorem}\label{maximum-principle}
Let $\Omega$ be open and $\Sigma_1$ and $\Sigma_2$ be connected
minimal surfaces in $\Omega$, and suppose that $\Sigma_1$ is
smooth.  Further suppose that $\Sigma_1$ lies weakly on one side of
$\Sigma_2$ (that is, $\Sigma_1$ is contained in the closure of one of the
components of $\Omega\backslash\Sigma_2$).  Then the two surfaces either
coincide (ie, are either the same surface or are each subsets of a single
larger minimal surface) or are disjoint.

The conclusion also holds when $\Sigma_2$ is mean-convex, that is,
its mean curvature vector points (where it is not $0$) into the
component of $\Omega\backslash\Sigma_2$ containing $\Sigma_1$.
\end{theorem}

\begin{proof} [Proof of Lemma~\ref{no-closed-minimal-surfaces}]
First, we observe that each surface $\Gamma_a=\{z=a\}$ with $a\neq
0,\pi$ is strictly mean convex, with mean curvature vector
pointing towards the strictly stable $2$-sphere $\Gamma$.  If
$0<a<\pi$, consider the unit normal vector
$\vec{n}_{\Gamma_a}=-\frac{\partial}{\partial z}$ pointing towards
$\Gamma$.  We can extend this vector to a vector field
$X=-\frac{\partial}{\partial z}$ defined on $M$, and then we can
compute the mean curvature of $\Gamma_a$ with respect to $\vec{n}_{\Gamma_a}$ as:
 \begin{eqnarray}
  H_{\Gamma_a} & = & -\frac{1}{2}\text{div}_MX \nonumber \\
   & = & \left.-\frac{1}{2}\frac{\partial}{\partial z}\right|_{z=a}(-\ln({\text{det}g_{\epsilon}})) \nonumber \\
   & = & \left.\frac{1}{2}\frac{\partial}{\partial z}\right|_{z=a}(\ln({\omega^4(z)\alpha^2(\phi)})) \nonumber \\
   & = & 2\frac{\omega'(a)}{\omega(a)}>0. \nonumber
 \end{eqnarray}
If $-\pi<a<0$, the same argument applied to the unit normal vector
$\vec{n}_{\Gamma_a}=\frac{\partial}{\partial z}$ shows that the
surface is mean convex.

Suppose now that $\Sigma$ is a closed minimal surfaces away from
$\{z=\pi\}$ and different from $\Gamma$.  Then there is a value
$0<a<\pi$ such that $\Sigma\subset\{-a\le z\le a\}$ and either
$\Sigma\cap\Gamma_a\neq\emptyset$ or
$\Sigma\cap\Gamma_{-a}\neq\emptyset$.  But this contradicts the
maximum principle (Theorem~\ref{maximum-principle}), as
$\Gamma_{\pm a}$ is strictly mean convex.
\end{proof}

The surfaces in our sequence will be embedded minimal cylinders
with boundary in $\{z=\pi\}$ that spiral an increasing number of
times around the polar axes $\gamma_1=\{\phi=0\}$ and
$\gamma_2=\{\phi=\pi\}$.  In order to find these surfaces, we wish
to solve Plateau problems in the universal cover $\Omega'$ of
$\Omega\backslash(\gamma_1\cup\gamma_2)$.  In particular, we will
use the following theorem of Meeks and Yau (Thm 5 of
\cite{meeks-yau}):

\begin{theorem}\label{plateau-existence-theorem}
Let $M$ be a compact Riemannian $3$-manifold, and suppose that $M$
can be isometrically embedded in a larger manifold $\tilde M$ such
that $\partial M$ is a $2$-dimensional subcomplex of $\tilde M$
consisting of smooth $2$-dimensional simplexes
$\{\Lambda_1,\ldots,\Lambda_k\}$ with the following properties:
\begin{enumerate}
\item Each $\Lambda_i$ is a smooth surface in $\tilde M$ whose
mean curvature is
  non-negative with respect to the outward normal.
\item Each surface $\Lambda_i$ is a compact subset of some smooth
surface
  $\tilde \Lambda_i\subset\tilde M$ where $\tilde \Lambda_i\cap M = \Lambda_i$ and $\partial
  \tilde \Lambda_i\subset\partial\tilde M$.
\end{enumerate}

Let $N$ be a compact subdomain of $\partial M$ such that each
homotopically nontrivial closed curve in $N$ is also homotopically
nontrivial in $M$.  Then there is a stable minimal embedding
$f:N\to M$ so that $f(\partial N)=\partial N$.
\end{theorem}

The stable minimal surface constructed in this theorem is, in fact, area-minimizing, as shown in the proof in \cite{meeks-yau}.  With the help of this theorem, we will construct a sequence of
embedded minimal disks that are smooth in the interior.  Then we
will apply the following result of Hardt and Simon to prove
regularity of the boundary (Theorem 11.1 in \cite{hardt-simon}):

\begin{theorem}\label{hardt-simon}
Let $U$ be an area-minimizing surface and $\partial U$ be a
connected oriented embedded $C^{1,\alpha}$ curve.  Then $U$ is a
connected embedded $C^{1,\alpha}$ hypersurface with boundary.
\end{theorem}

This construction gives us half of the cylinders $\Sigma_n$.  Then,
we will use the Schwarz reflection principle (Lemma 7.3 in
\cite{osserman}) to reflect them across the axes $\gamma_1$ and
$\gamma_2$:

\begin{theorem}\label{reflection-principle}
Suppose $\Sigma$ is a minimal surface in a $3$-manifold $M$ whose
boundary contains a geodesic segment $\gamma$.  Suppose
additionally that there exists an isometry $G$ in $M$ whose fixed
points include $\gamma$.  Then $\Sigma$ can be extended to a
minimal surface $\Sigma\cup G(\Sigma)$ symmetric across $\gamma$.
\end{theorem}

We recall that a codimension one {\it lamination} on a $3$-manifold $M^3$ is a collection $\mathcal{L}$ of smooth disjoint surfaces (called leaves) such that $\bigcup_{\Lambda\in\mathcal{L}}\Lambda$ is closed, and such that for each $x\in M$ there exists an open neighborhood $U$ of $x$ and a coordinate chart $(U,\Phi)$ with $\Phi(U)\subset\R^3$ so that in these coordinates the leaves in $\mathcal{L}$ pass through $\Phi(U)$ in slices of the form $(\R\times\{t\})\cap\Phi(U)$.  A {\it minimal lamination} is a lamination whose leaves are minimal.  Note that any (compact) embedded surface is a lamination.  When proving the convergence of our sequence, we will need the following proposition (Proposition B.1 in \cite{colding-minicozzi3}):

\begin{proposition} \label{minimal-lamination}
Let $M^3$ be a fixed $3$-manifold.  If $\mathcal{L}_i\subset B_{2R}(x)\subset M$ is a sequence of minimal laminations with uniformly bounded curvatures (where each leaf has boundary contained in $\partial B_{2R}(x)$), then a subsequence $\mathcal{L}_j$ converges in the $C^{\alpha}$ topology for any $\alpha<1$ to a (Lipschitz) lamination $\mathcal{L}$ in $B_R(x)$ with minimal leaves.
\end{proposition}

\section{The sequence of embedded minimal surfaces $\{\Sigma_n\}$}
\label{sect2}

We consider $\Omega'$ the universal cover of
$\Omega\backslash(\gamma_1\cup\gamma_2)$.  We can parametrize this
universal cover as $\Omega'=[0,\pi]\times\R\times(-\pi,\pi)$.  In these
coordinates, the metric in $\Omega'$ is given by the same expression as
the original metric in $M$.  The topological boundary of $\Omega'$ is
$\partial\Omega'=\{\phi=0\}\cup\{\phi=\pi\}\cup\{z=\pi\}\cup\{z=-\pi\}$,
and the two pieces $\{z=\pi\}$ and $\{z=-\pi\}$ are mean-convex.

The main difficulty in applying
Theorem~\ref{plateau-existence-theorem} is the portion of the boundary
corresponding to $\gamma_1$ and $\gamma_2$, as the warped product metric
in $\Omega'$ becomes degenerate as we approach these axes.  In order to
apply Meeks-Yau, we will remove small
$\epsilon$-tubes around $\gamma_1$ and $\gamma_2$, then smoothly glue
together these new ends.  The resulting manifold will be
topologically a solid torus and will only have mean-convex
boundary.  Thus, we can solve Plateau problems in a sequence of
these manifolds with $\epsilon\rightarrow 0$.  Subsets of the
area-minimizing solutions in this sequence can be isometrically embedded in
$\Omega'$, and projecting back to $\Omega$ we can get a subsequence converging to an
embedded solution to the original Plateau problem in $\Omega'$.

The precise construction is as follows.

For each $\epsilon>0$, let $\Omega'_{\epsilon}$ be the universal
cover of $\Omega$ minus the $\epsilon$-tubes around $\gamma_1$ and
$\gamma_2$.
Thus,$$\Omega'_{\epsilon}=(\epsilon,\pi-\epsilon)\times\R\times(-\pi,\pi)\subset\Omega'.$$
We can isometrically embed $\Omega'_{\epsilon}$ into the manifold
$N_{\epsilon}=\mathbf{S}^1\times\R\times[-\pi,\pi]$ endowed with the metric
$$g_{\epsilon}^2 = \omega^2(z)(d\phi^2 + \alpha_\epsilon^2(\phi)\;d\theta^2)+dz^2,$$ where $\alpha_\epsilon:\S^1=[0,\pi)\rightarrow\R$ is a smooth function satisfying:
\begin{enumerate}
 \item $\alpha_\epsilon(\phi)=\sin(\phi)$ for $\phi\in[\epsilon,\pi-\epsilon]$.
 \item $\alpha_\epsilon'(\phi)>0$ for $0<\phi<\frac{\pi}{2}$ and $\alpha_\epsilon'(\phi)<0$ for $\frac{\pi}{2}<\phi<\pi$.
 \item $\alpha_\epsilon(\phi)>0$ for all $\phi\in\S^1$.
\end{enumerate}
Conceptually, we imagine $\alpha_\epsilon(\phi)$ as a smooth version of
$f(\phi) = \max\{\sin(\phi), \sin(\epsilon)\}$.  This metric glues
together the $\epsilon$-tubes around $\gamma_1$ and $\gamma_2$.

We next construct embedded minimal surfaces in
$N_{n,\epsilon}=\mathbf{S}^1\times[-n\pi,n\pi]\times[-\pi,\pi]\subset
N_{\epsilon}$ with boundary $\gamma_3=\{(\phi,-n\pi,-\pi)\}$ and
$\gamma_4=\{(\phi,n\pi,\pi)\}$ using
Theorem~\ref{plateau-existence-theorem}.  To do this, we first
need to show that the boundary of $N_{n,\epsilon}$ is mean convex:

\begin{lemma}
The boundary $\partial N_{n,\epsilon}$ satisfies the hypothesis of Theorem~\ref{plateau-existence-theorem}.  In particular, it is contained in the union
of a finite number of minimal surfaces that meet at an angle
of $\frac{\pi}{2}$.
\end{lemma}

\begin{proof}
 First, observe that the boundary of $N_{n,\epsilon}$ consists of the
 ``horizontal'' surfaces $\Lambda_1=\{z=-\pi\}$ and $\Lambda_2=\{z=\pi\}$
 and the ``vertical'' surfaces $\Lambda_3=\{\theta=-n\pi\}$ and
 $\Lambda_4=\{\theta=n\pi\}$, and neither the two vertical pieces nor the two
horizontal pieces intersect.  Therefore, to satisfy the
intersection requirement, we only need to look at the intersection
of a horizontal surface and a vertical surface.

Let $x\in\Lambda_1\cap\Lambda_3$.  Then $\frac{\partial}{\partial\phi}\in T_x(\Lambda_1\cap\Lambda_3)$, so the angle between the surfaces is the angle between $\frac{\partial}{\partial\theta}\in T_x(\Lambda_1)$ and $\frac{\partial}{\partial z}\in T_x(\Lambda_3)$.  Direct computation shows that these two vectors are orthogonal, and so the surfaces meet orthogonally.  The same argument shows that the other three pairs of surfaces also intersect orthogonally.  In particular, this shows that if we embed $N_{n,\epsilon}$ in a larger manifold and extend each $\Lambda_i$, the extension will lie outside of $N_{n,\epsilon}$, as required in (ii) in Theorem~\ref{plateau-existence-theorem}.

We will next show that the horizontal surfaces $\Lambda_1$ and
$\Lambda_2$ are minimal and thus mean convex.

The unit normal vector of $\Lambda_1$ that points towards the
interior of $N_{n,\epsilon}$ is
$\vec{n}_{\Lambda_1}=\frac{\partial}{\partial z}$.  We can extend
this vector to the coordinate vector field
$X=\frac{\partial}{\partial z}$ on $N_{n,\epsilon}$, and then we
can compute the mean curvature vector of $\Lambda_1$ with respect to $\vec{n}_{\Lambda_1}$ as:
 \begin{eqnarray}
  H_{\Lambda_1} & = & -\frac{1}{2}\text{div}_{N_{n,\epsilon}}X \nonumber \\
   & = & \left.-\frac{1}{2}\frac{\partial}{\partial z}\right|_{z=-\pi}\ln({\text{det}g_{\epsilon}}) \nonumber \\
   & = & \left.-\frac{1}{2}\frac{\partial}{\partial z}\right|_{z=-\pi}(\ln({\omega^4(z)\alpha^2(\phi)})) \nonumber \\
   & = & -2\frac{\omega'(-\pi)}{\omega(-\pi)}=0 \nonumber
 \end{eqnarray}
 This shows that $\Lambda_1$ is minimal in $N_{n,\epsilon}$, and in
 particular it is mean convex.  A similar argument shows that $\Lambda_2$ is mean convex.

 Finally, we will show that the vertical surfaces $\Lambda_3$ and $\Lambda_4$ are minimal in $N_{n,\epsilon}$.
 The inward-pointing unit normal of $\Lambda_3$ is $\vec{n}_{\Lambda_3}=\frac{1}{\omega(z)\alpha(\phi)}\frac{\partial}{\partial\theta}$, which can be extended to a vector field $X=\frac{1}{\omega(z)\alpha(\phi)}\frac{\partial}{\partial\theta}$ in all $N_{n,\epsilon}$.  Then the mean curvature vector of $\Lambda_3$ with respect to $\vec{n}_{\Lambda_3}$ is:
 \begin{eqnarray}
  H_{\Lambda_3} & = & -\frac{1}{2}\text{div}_{N_{n,\epsilon}}X \nonumber \\
   & = & \left.-\frac{1}{2}\frac{\partial}{\partial\theta}\right|_{\theta=-n\pi}\ln({\text{det}g_{\epsilon}})=0 \nonumber
 \end{eqnarray}
 since the determinant of the metric does not depend on the $\theta$
 coordinate.  This proves that $\Lambda_3$ is minimal, and a similar
 calculation shows that $\Lambda_4$ is also minimal.

Therefore, the first hypothesis in Theorem~\ref{plateau-existence-theorem} is also satisfied.
\end{proof}

Consider the circles $\gamma_3 = \{\theta=-n\pi, z=-\pi\}$ and $\gamma_4 =
\{\theta=n\pi,z=\pi\}$: they bound an annulus in
$\partial N_{n,\epsilon}$.  Moreover, any homotopically nontrivial curve in that annulus is also homotopically nontrivial in $N_{n,\epsilon}$, so by
Theorem~\ref{plateau-existence-theorem}, there exists a smoothly
embedded area-minimizing annulus $\Sigma''_{n,\epsilon}$ in
$N_{n,\epsilon}$ whose boundary is $\gamma_3\cup\gamma_4$.  These
surfaces have the following non-intersection property:

\begin{lemma}\label{intersection-lemma}
Let $R_{\theta_0}:N_{\epsilon}\to N_{\epsilon}$ be the
isometry $\theta\to\theta+\theta_0$, where $\theta_0\neq 0$.  Then
$$R_{\theta_0}(\Sigma''_{n,\epsilon})\cap \Sigma''_{n,\epsilon}=\emptyset.$$
\end{lemma}

\begin{proof}[Proof of Lemma~\ref{intersection-lemma}]
For $\theta$ big enough (for instance, $\theta>2n\pi$), it is clear that $R_{\theta}(\Sigma''_{n,\epsilon})\cap \Sigma''_{n,\epsilon}=\emptyset$.  Let $$\theta_0=\sup\{\theta>0:R_{\theta}(\Sigma''_{n,\epsilon})\cap \Sigma''_{n,\epsilon}\neq\emptyset\}.$$ Then $R_{\theta_0}(\Sigma''_{n,\epsilon})\cap\Sigma''_{n,\epsilon}\neq\emptyset$ and $R_{\theta_0}(\Sigma''_{n,\epsilon})$ lies on one side of $\Sigma''_{n,\epsilon}$.  By the maximum principle, Theorem~\ref{maximum-principle}, this implies that $R_{\theta_0}(\Sigma''_{n,\epsilon})=\Sigma''_{n,\epsilon}$, which means $\theta_0=0$.  Therefore, $R_{\theta}(\Sigma''_{n,\epsilon})\cap \Sigma''_{n,\epsilon}=\emptyset$ for all $\theta>0$.  The same argument applies to $\theta<0$.

\end{proof}

Define $\Sigma'_{n,\epsilon}$ to be the portion of $\Sigma''_{n,\epsilon}$
contained in $\Omega'_{\epsilon}$.  This discards the portion of our
minimal surfaces that lies inside of the $\epsilon$-tubes, where the metric
in $N_{n,\epsilon}$ is different than in $\Omega'$, the universal cover of
$\Omega\backslash(\gamma_1\cup\gamma_2)$.  Then $\Sigma'_{n,\epsilon}$ is
an area-minimizing surface with boundary consisting of
$\gamma_3\cup\gamma_4\cup\delta^1_{n,\epsilon}\cup\delta^2_{n,\epsilon}$,
where each $\delta^i_{n,\epsilon}$ is some curve on the
cylinder of radius $\epsilon$ centered on the axes $\gamma_i$.
Let $\Sigma_{n,\epsilon}\subset\Omega$ be the projection of
$\Sigma'_{n,\epsilon}\subset\Omega'$ from the universal cover to the
original manifold $\Omega\subset\mathbf{S}^2\times\mathbf{S}^1$.

By Lemma~\ref{intersection-lemma}, $\Sigma''_{n,\epsilon}$ does
not intersect any of its $\theta$-translations.  Because a
self-intersection of $\Sigma_{n,\epsilon}$ corresponds to the
intersection of $\Sigma'_{n,\epsilon}\subset\Sigma''_{n,\epsilon}$
with one of its $\theta$-translations, $\Sigma_{n,\epsilon}$ must
be embedded in $\Omega$.  More strongly, the intersection of
$\Sigma_{n,\epsilon}$ with a plane $\{\phi=\mathrm{const}\}$ must be
monotone in the vertical coordinate $z$.

At this point in our construction, the surfaces $\Sigma_{n,\epsilon}$ are
essentially embedded minimal surfaces in $\Omega$ that wrap $n$ times
around $\epsilon$-tubes centered on the axes $\gamma_1$ and $\gamma_2$ and
whose boundary is the union
of fixed geodesics $\gamma_3$, $\gamma_4$, and some curves contained in these
$\epsilon$-tubes.  The next step is to
let $\epsilon\to 0$ to create a sequence of minimal surfaces $\tilde{\Sigma}_n$
that wrap $n$ times around the axes $\gamma_1$ and $\gamma_2$ and whose
boundaries are $\cup_{i=1}^4 \gamma_i$.  These surfaces are essentially
``half-helicoids'' which we will extend via Schwarz reflection
(Theorem~\ref{reflection-principle}) into our
final sequence $\Sigma_n$.

Specifically, to construct each $\tilde{\Sigma}_n$, we need area and curvature
bounds on $\Sigma_{n,\epsilon}$ independent of $\epsilon$.  First we
obtain a curvature bound: let
$\Omega_{\epsilon_0}$ be the complement of the
$\epsilon_0$-neighborhood of $\cup_{i=1}^4 \gamma_i$.  Clearly,
$\cup_{\epsilon_0>0}\Omega_{\epsilon_0}=\Omega\backslash(\cup_{i=1}^4
\gamma_i)$.  Fix $\epsilon_0>0$.  Because each $\Sigma_{n,\epsilon}$ is
isometric to a subset of the stable minimal surfaces
$\Sigma''_{n,\epsilon}$, we have the interior curvature bound
$$\sup_{\Omega_{\epsilon_0}}|A(\Sigma_{n,\epsilon})|^2 \leq
C(\epsilon_0),$$ for any $\epsilon<\epsilon_0/2$ and some constant
$C(\epsilon_0)$ dependent only on $\epsilon_0$.

Similarly, we can obtain a uniform area bound for $\Sigma_{n,\epsilon}\cap\Omega_{\epsilon_0}$.  Each
$\Sigma_{n,\epsilon}$ is the projection of an area-minimizing surface
$\Sigma'_{n,\epsilon}\subset\Omega'$.  As before, let
$\Omega'_{\epsilon_0}$ be the universal cover of $\Omega_{\epsilon_0}$ (or
equivalently, $\Omega'_{\epsilon_0}$ is the complement in $\Omega'$ of the
$\epsilon_0$-neighborhoof of $\partial\Omega'$).  If we construct a
surface consisting of a fixed surface
$\Sigma'_{n,\epsilon_0}\cap\Omega'_{\epsilon_0}$ plus the piece of
$\partial\Omega'_{\epsilon_0}$ between
$\partial\Sigma'_{n,\epsilon}\cap\Omega'_{\epsilon_0}$ and
$\partial\Sigma'_{n,\epsilon_0}\cap\Omega'_{\epsilon_0}$, then this new
surface has bigger area than
$\Sigma'_{n,\epsilon}\cap\Omega'_{\epsilon_0}$.  We can then bound
uniformly the area of $\Sigma'_{n,\epsilon}\cap\Omega'_{\epsilon_0}$ by
the area of $\Sigma'_{n,\epsilon_0}\cap\Omega'_{\epsilon_0}$ plus the area
of $\partial\Omega'_{\epsilon_0}$.  Now since the area of the projection
$\Sigma_{n,\epsilon}$ is the same as the area of $\Sigma'_{n,\epsilon}$
(because $\Omega$ and $\Omega'$ are isometric and $\Sigma_{n,\epsilon}$ is
embedded, that is, there is no self-intersections), this gives us a bound
of the area of $\Sigma_{n,\epsilon}$ uniform in $\epsilon$.

This area bound combined with the previous interior curvature bound
implies that for any fixed $\epsilon_0>0$, we can find a sequence of
$\Sigma_{n,\epsilon}$ that converges when $\epsilon\to 0$ to a locally area-minimizing embedded
surface in $\Omega_{\epsilon_0}$.  By letting $\epsilon_0\to 0$
and taking a diagonal sequence, we can extract a subsequence
$\{\Sigma_{n,\epsilon_j}\}$ that converges to an embedded minimal
surface $\tilde\Sigma_n$ in
$\Omega\backslash(\cup_{i=1}^4\gamma_i)$.  Additionally,
$\partial\Sigma_{n,\epsilon_j}$ converges pointwise to
$\cup_{i=1}^4\gamma_i$.

So far, we have only shown that $\tilde\Sigma_n$ is embedded away from
$\gamma_1$ and $\gamma_2$.  However, consider the portion of
$\partial\Sigma_{n,\epsilon_j}$ that lies on the cylinders
$\{\phi=\epsilon_j,\pi-\epsilon_j\}$.  By Lemma~\ref{intersection-lemma},
the projection of these portions onto $\gamma_1$ and $\gamma_2$ must be
monotone (otherwise, $\Sigma_{n,\epsilon_j}'$ would intersect one of its
$\theta$-translations; see Figure~\ref{picJK}).  In particular, this
implies that $\tilde\Sigma_n$ is embedded up to the boundary
$\cup_{i=1}^4\gamma_i$.

\begin{figure}[ht]
 \begin{center}
 \includegraphics[width=3in]{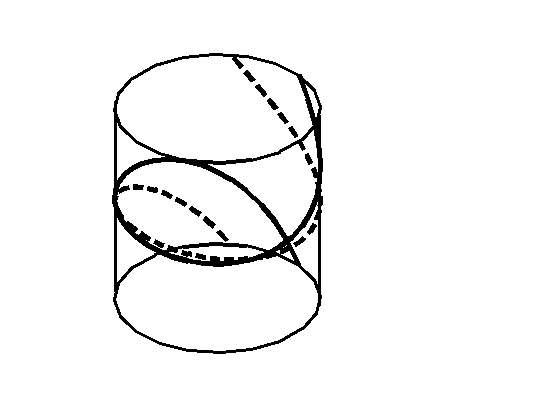}
 \includegraphics[width=2.5in]{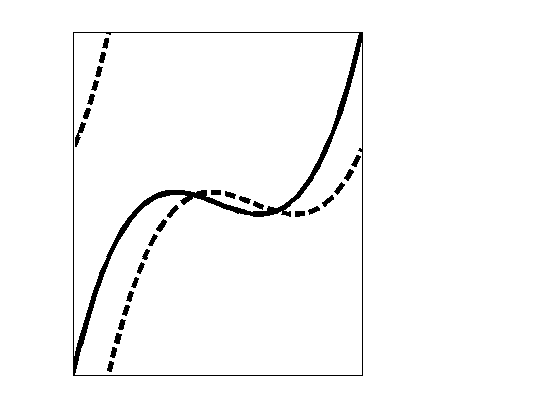}
 \caption[Monotonicity of the boundary curve.]{If the boundary curve $\Sigma_{n,\epsilon_j}\cap\{\phi=\epsilon_j,\pi-\epsilon_j\}$ is not monotone, then it intersects its $\theta$-translation, that is, its rotation around the polar axes.}\label{picJK}
\end{center}
\end{figure}

To prove that this embedding is non-singular at a point
$x\in\gamma_1\cup\gamma_2$, choose a small intrinsic neighborhood
$U\subset\tilde\Sigma_n$ of $x$.  As $\tilde\Sigma_n$ is the limit of the
area-minimizing surfaces $\Sigma_{n,\epsilon}$, $U$ will
be area-minimizing provided that it is sufficiently small.  We can also
choose $U$ to have smooth boundary.  Therefore, we can apply
Theorem~\ref{hardt-simon} to show that $U$ is $C^{1,\alpha}$ up to
$\partial U\cap\{\gamma_1\cup\gamma_2\}$.

Thus, the surfaces $\tilde{\Sigma}_n$ have boundary on
$\{z=\pm\pi\}\cup\gamma_1\cup\gamma_2$.  The final step is to use
Theorem~\ref{reflection-principle} to reflect them across the axes
$\gamma_1$ and $\gamma_2$.  In this case the isometry $G$ is the
rotation $R_\pi:\theta\mapsto\theta+\pi$ around the polar axes.
The new surfaces $\Sigma_n$ are cylinders in $\Omega$ with
boundary in $\partial\Omega=\{z=-\pi\}\cup\{z=\pi\}$.  Observe
that $\Sigma_n\backslash(\gamma_1\cup\gamma_2)=\tilde{\Sigma}_n\cup
G(\tilde{\Sigma}_n)$, and when we lift $G(\tilde{\Sigma}_n)$ to
$\Omega'$ we obtain just a $\theta$-translation of $\Sigma'_n$.  By
Lemma~\ref{intersection-lemma}, we can show that
$\tilde{\Sigma}_n\cap G(\tilde{\Sigma}_n)\subset\gamma_1\cup\gamma_2$, and
therefore each $\Sigma_n$ is a smooth embedded minimal surface in
$\Omega$.

\vskip6mm
\section{The limit of the sequence $\Sigma_n$} \label{sect3}

We will now finish the proof of Theorem~\ref{thm1} by showing that
the sequence $\Sigma_n$ converges smoothly to a minimal lamination
away from the axes $\gamma_1$ and $\gamma_2$ and that the limit
lamination has two helicoid-like singularities on the strictly
stable $2$-sphere $\Gamma$.

We say that a sequence of surfaces $\Sigma_n\subset M$ is
uniformly locally simply connected (ULSC) if for each $x\in M$,
there exists a constant $r_0>0$ (depending on $x$), so that for
all $r\leq r_0$ and every surface $\Sigma_n$ each connected
component of $B_r(x)\cap\Sigma_n$ is a disk.

The existence of the limit lamination is given by the following
lemma (see \cite{colding-minicozzi3}, Lemma II.1.2 in \cite{colding-minicozzi10} and Theorem 4.2 in \cite{colding-minicozzi11}):

\begin{lemma}\label{lamination-lemma}
Let $M$ be a compact $3$-dimensional manifold with boundary and
let $\Sigma_n$ be a sequence of compact embedded minimal surfaces
in $M$ with $\partial\Sigma_n\subset\partial M$.  Let
$$S = \left\{ x \in M : \sup_{B_r(x)\cap\Sigma_j}|A|^2\to\infty
\mbox{ for all } r>0\right\}$$
and suppose the sequence is ULSC in a neighborhood
of $S$.  Then after passing to a subsequence,
$\mathring{\Sigma}_n\backslash S$ will converge on compact subsets to a
lamination $\mathcal{L}\subset \mathring{M}\backslash S$ with minimal
leaves.
\end{lemma}

\begin{proof}
For each compact subset $K\in M\backslash S$, there is an open covering of
$K$ by finitely many balls where the curvatures of the $\Sigma_j$'s are
bounded uniformly in $j$ in the concentric double balls (by the definition
of $S$).  The claim then follows from Proposition~\ref{minimal-lamination}
and a diagonal argument.
\end{proof}

To apply Lemma~\ref{lamination-lemma} to our case, we need to show
that $\{\Sigma_n\}$ is ULSC in a neighborhood of the singular set
$S$.  Since the ``half-helicoids'' $\tilde{\Sigma}_n$ are stable,
we have uniform curvature bounds on $\Sigma_n$ depending only on the
distance to the boundary and the axes $\gamma_1$ and $\gamma_2$.
By slightly shrinking $\Omega$, we can obtain uniform
curvature bounds away from the polar axes.  Therefore the singular
set $S$ must be contained in $\gamma_1\cup\gamma_2$.

A point $y$ in the singular set $S$ is not ULSC if and only if the
injectivity radius on the sequence go to $0$ at that point, that is, there
is a sequence of points $y_n\in\Sigma_n$ with $y_n\to y$ and $\mathrm{inj}(y_n)\to
0$ (see Part IV in \cite{colding-minicozzi10}, Proposition I.0.19 in
\cite{colding-minicozzi8}).  If $d_{\Sigma_n}(y_n,y)$ is bounded away from
zero (ie, the sequence $y_n$ converges to $y$ extrinsically but not
intrinsically), then standard curvature bounds for stable minimal surfaces
(see, for example, Theorem 2.7 of \cite{colding-minicozzi})
imply that the injectivity radius of $y_n$ is bounded away from zero.
Therefore, we only need to prove a lower bound on injectivity radius for
all points in an intrinsic neighborhood of the axes.

Consider an intrinsic $r$-neighborhood of the axis
$D_r^n(\gamma_1)=\{x\in\Sigma_n:d_{\Sigma_n}(x,\gamma_1)<r\}.$  We
can write  $D^n_r(\gamma_1)=\tilde{D}^n_r(\gamma_1)\cup
G(\tilde{D}^n_r(\gamma_1))$, where $\tilde{D}^n_r(\gamma_1)$ is
the corresponding $r$-neighborhood of $\gamma_1$ in
$\tilde{\Sigma}_n$, and $\tilde{D}^n_r(\gamma_1)\cap
G(\tilde{D}^n_r(\gamma_1))=\gamma_1$.  Since both
$\tilde{\Sigma}_n$ and $G(\tilde{\Sigma}_n)$ are disks, $\tilde{D}^n_r(\gamma_1)\backslash\gamma_1$ and
$G(\tilde{D}^n_r(\gamma_1))\backslash\gamma_1$ are disjoint disks, and
therefore $D^n_r(\gamma_1)$ is topologically a disk.  Thus, in smaller
neighborhood $D_{\frac{r}{2}}(\gamma_1)$, the injectivity radius is
bounded below by $\frac{r}{2}$, and this implies that the sequence is ULSC
in the axis $\gamma_1$.  The same argument applies to $\gamma_2$.

Once we have the minimal lamination $\mathcal{L}$, there are only
two possibilities for each leaf; either the leaf is a closed
surface or it spirals infinitely into a closed surface.  The only
closed surface in $\Omega$ is the strictly stable $2$-sphere
$\Gamma$, and so our lamination will consist of $\Gamma$ and two
minimal surfaces that spiral into it.

\begin{theorem}\label{proper-theorem}
After passing to a subsequence, the sequence $\{\Sigma_n\}$
converges to a lamination $\mathcal{L}$ of minimal surfaces
consisting of the strictly stable $2$-sphere $\Gamma$ and two
embedded minimal surfaces that lie on opposite sides of and spiral
into $\Gamma$.
\end{theorem}

\begin{proof}
By Lemma~\ref{lamination-lemma}, a subsequence of $\{\Sigma_n\}$
will converge to a minimal lamination $\mathcal{L}$ away from
$\gamma_1\cup\gamma_2$.  Let $\Sigma$ be one of the leaves of
$\mathcal{L}$.  Define the set $\{(\phi,\theta,z): z =
\inf_{(\phi,\theta,\zeta)\in\Sigma}\zeta\}$.  Because $\Sigma$ is
embedded and $\mathcal{L}$ is a (closed) lamination, this set
defines a leaf of $\mathcal{L}$ which is a closed minimal graph over $\Gamma$ in $M$.  By
Lemma~\ref{no-closed-minimal-surfaces}, there are no non-trivial
minimal graphs over $\Gamma$ in $\Omega$, which implies that
$\Sigma=\Gamma$ or that $\Sigma$ must infinitely spiral into
$\Gamma$.  In either case, this shows that $\Gamma\in\mathcal{L}$.

Because each $\Sigma_n$ has boundary in both boundary components
of $\Omega$, the lamination $\mathcal{L}$ must contain at least
one leaf that does not equal to $\Gamma$ and therefore must spiral
into $\Gamma$.  Since each $\Sigma_n$ is connected and crosses
$\Gamma$ but no leaf of $\mathcal{L}$ can cross $\Gamma$, there
must be exactly two such spiraling surfaces, one on each side of
$\Gamma$.
\end{proof}

Theorem~\ref{thm1} now follows immediately.

\begin{proof}[Proof of Theorem~\ref{thm1}]
We can calculate the scalar curvature of a warped product on $M$
as
$$\mathrm{Scal}_M=-2\frac{\omega''}{\omega}+\frac{1-(\omega')^2}{\omega^2}.$$
In particular, we can choose $\omega\in C^{\infty}(\mathbf{S}^1)$
satisfying the four properties listed before and with positive
scalar curvature.  An example of such a function would be
$$\omega(z)=-\frac{1}{4}\cos z+\frac{5}{4}.$$
\end{proof}


\begin{thebibliography}{99}

\bibitem{alarcon-ferrer-martin} {\bf A. Alarcon}, {\bf L. Ferrer}, and {\bf F. Martin}, Density theorems for complete minimal surfaces in $\mathbf{R}^3$, http://arxiv.org/math.DG/0603737

\bibitem{baldi} {\bf A. Baldi}, Weighted BV functions.  Houston J. Math
  {\bf 27} (2001), no. 3, 683--705.

\bibitem{choi-schoen} {\bf H. Choi} and {\bf R. Schoen}, The space of
minimal embeddings of a surface into a three-dimensional manifold of
positive Ricci curvature.  Invent. Math. {\bf 81} (1985), no. 3, 387--394.

\bibitem{choi-wang} {\bf H. Choi} and {\bf A. Wang}, A first eigenvalue
estimate for minimal hypersurfaces.  J. Differential Geom. {\bf 18} (1983),
no. 3, 559--562.

\bibitem{colding-delellis} {\bf T. H. Colding} and {\bf C. De Lellis},
Singular limit laminations, Morse index, and positive scalar
curvature. Topology {\bf 44} (2005), no. 1, 24--45.

\bibitem{colding-delellis-minicozzi} {\bf T. H. Colding}, {\bf C. De Lellis},
and {\bf W. P. Minicozzi}.  Three circles theorems for Schr\"{o}dinger
operators on cylindrical ends and geometric applications.  Preprint.

\bibitem{colding-minicozzi} {\bf T. H. Colding} and {\bf W. P. Minicozzi}, {\it
Minimal Surfaces}, Courant Lecture Notes in Mathematics, 4.  New York
University, Courant Institute of Mathematical Sciences, New York, 1999.

\bibitem{colding-minicozzi2} {\bf T. H. Colding} and {\bf W. P. Minicozzi},
Embedded minimal disks: proper versus nonproper--global versus local.
Trans. Amer. Math. Soc. {\bf 356} (2004), no. 1, 283--289.

\bibitem{colding-minicozzi11} {\bf T. H. Colding} and {\bf W. P. Minicozzi}, Embedded minimal surfaces without area bounds in $3$-manifolds. Geometry and Topology: Aarhus (1998), Contemp. Math. {\bf 258}, Amer. Math. Soc., Providence, RI, 2000.

\bibitem{colding-minicozzi9} {\bf T. H. Colding} and {\bf W. P. Minicozzi},
  Finiteness of Embedded Minimal Surfaces in $3$-Manifolds with Positive
  Scalar Curvature.  In Preparation.

  \bibitem{colding-minicozzi4} {\bf T. H. Colding} and {\bf W. P. Minicozzi}, Multivalued minimal graphs and properness of disks. Int. Math. Res. Not 2002, no. 21, 1111--1127.

\bibitem{colding-minicozzi5} {\bf T. H. Colding} and {\bf W. P. Minicozzi}, The
Calabi-Yau conjectures for embedded surfaces. Ann. of Math. (2), to appear.


\bibitem{colding-minicozzi6} {\bf T. H. Colding} and {\bf W. P. Minicozzi}, The space of embedded minimal surfaces of fixed genus in a $3$-manifold. I. Estimates off the axis for disks. Ann. of Math. (2) {\bf 160} (2004), no. 1, 27--68.

\bibitem{colding-minicozzi7} {\bf T. H. Colding} and {\bf W. P. Minicozzi}, The space of embedded minimal surfaces of fixed genus in a $3$-manifold. II. Multi-valued graphs in disks. Ann. of Math. (2) {\bf 160} (2004), no. 1, 69--92.

\bibitem{colding-minicozzi8} {\bf T. H. Colding} and {\bf W. P. Minicozzi}, The space of embedded minimal surfaces of fixed genus in a $3$-manifold. III. Planar domains. Ann. of Math. (2) {\bf 160} (2004), no. 2, 523--572.

\bibitem{colding-minicozzi3} {\bf T. H. Colding} and {\bf W. P. Minicozzi}, The
space of embedded minimal surfaces of fixed genus in a
$3$-manifold. IV. Locally simply connected.  Ann. of Math. (2) {\bf 160} (2004),
no. 2, 573--615.

\bibitem{colding-minicozzi10} {\bf T. H. Colding} and {\bf W. P. Minicozzi}, The
space of embedded minimal surfaces of fixed genus in a
$3$-manifold. V. Fixed genus. Preprint, 2005.
http://www.arxiv.org/pdf/math/0509647


\bibitem{collin-rosenberg} {\bf P. Collin} and {\bf H. Rosenberg}, Notes sur la d\'{e}monstration de N. Nadirashvili des conjectures de Hadamard et Calabi-Yau. Bull. Sci. Math. {\bf 123} (1999), no. 7, 563--575.

\bibitem{federer} {\bf H. Federer}, The singular sets of area minimizing
  rectifiable current with codimension one and of area minimizing flat
  chains modulo two with arbitrary codimension, Bull. A.M.S. {\bf 76}
  (1980), 767--771.

\bibitem{gilbarg-trudinger} {\bf D. Gilbarg} and {\bf N. Trudinger}, {\it Elliptic Partial Differential Equations of Second Order}. Springer-Verlag, New York, 2nd Ed., 1983.

\bibitem{giusti} {\bf E. Giusti}, {\it Minimal Surfaces and Functions of
  Bounded Variation}. Monographs in Mathematics, 80. Birkh\"{a}user, Basel, 1984.

\bibitem{lin} {\bf Q. Han} and {\bf F. Lin}, {\it Elliptic Partial
  Differential Equations}, Courant Lecture Notes in Mathematics 1.  New
  York University, Courant Institute of Mathematical Sciences, New York,
  1997.

\bibitem{hardt-simon} {\bf R. Hardt} and {\bf L. Simon}, Boundary
  regularity and embedded solutions for the oriented Plateau problem.
  Ann. of Math. (2) {\bf 110} (1979), no. 3, 439--486.

\bibitem{hass-norbury-rubinstein} {\bf J. Hass}, {\bf P. Norbury}, and
  {\bf J. Rubinstein}, Minimal spheres of arbitrarily high Morse index.
  Comm. Anal. Geom. {\bf 11} (2003), no. 3, 425--439.

  \bibitem{hass-scott} {\bf J. Hass} and {\bf P. Scott },The existence of least area surfaces in $3$-manifolds. Trans. Amer. Math. Soc. {\bf 310} (1988), no. 1, 87--114.

\bibitem{hoffman-weber-wolf1} {\bf D. Hoffman}, {\bf M. Weber}, and {\bf
  M. Wolf}, An embedded genus one helicoid.  Ann. of Math. (2), to appear.

\bibitem{hoffman-weber-wolf2} {\bf D. Hoffman}, {\bf M. Weber}, and {\bf
M. Wolf}, An embedded genus-one helicoid.  PNAS 2005 102: 16566--16568.

\bibitem{hoffman-white} {\bf D. Hoffman} and {\bf B. White}, Genus-one helicoids from a variational point of view. Commentarii Mathematici Helvetici, to appear.

\bibitem{hsiang-lawson} {\bf W. Hsiang} and {\bf H. Lawson}, Minimal
  submanifolds of low cohomogeneity.  J. Differential Geometry {bf 5} (1971) 1-38.

\bibitem{jorge-xavier} {\bf L. Jorge} and {\bf F. Xavier}, A complete minimal surface in $\mathbf{R}^3$ between two parallel planes. Ann. of Math. (2) {\bf 112} (1980), no. 1, 203--206.

\bibitem{martin-morales} {\bf F. Martin} and {\bf S. Morales}, Complete proper minimal surfaces in convex bodies of $\mathbf{R}^3$. Duke Math. J. {\bf 128} (2005), no. 3, 559--593.

\bibitem{meeks} {\bf W. Meeks}, Regularity of the singular set in the Colding-Minicozzi lamination theorem. Duke Math. J. {\bf 123} (2004), no. 2, 329--334.

\bibitem{meeks2} {\bf W. Meeks}, The limit lamination metric for the Colding-Minicozzi minimal lamination. Illinois J. Math. {\bf 49} (2005), no. 2, 645--658.

\bibitem{meeks-perez} {\bf W. Meeks} and {\bf J. Perez}, Conformal properties in classical minimal surface theory. Surveys in differential geometry. Vol. IX, 275--335, Surv. Differ. Geom., IX, Int. Press, Somerville, MA, 2004.

\bibitem{meeks-rosenberg} {\bf W. Meeks}  and {\bf H. Rosenberg}, The uniqueness of the helicoid. Ann. of Math. (2) {\bf 161} (2005), no. 2, 727--758.

\bibitem{meeks-weber} {\bf W. Meeks} and {\bf M. Weber}, Bending the
helicoid. Preprint, 2005. http://arxiv.org/abs/math.DG/0511387

\bibitem{meeks-yau} {\bf W. Meeks} and {\bf S.T. Yau}, The Existence of
Embedded Minimal Surfaces and the Problem of Uniqueness. Math. Z. {\bf 179}
(1982), no. 2, 151--168.

\bibitem{morgan} {\bf F. Morgan}, Regularity of isoperimetric
  hypersurfaces in Riemannian manifolds.  Trans. Amer. Math. Soc. {\bf
  355} (2003), NO. 12, 5041--5052.

\bibitem{nadirashvili} {\bf N. Nadirashvili}, Hadamard's and Calabi-Yau's conjectures on negatively curved and minimal surfaces. Invent. Math. {\bf 126} (1996), no. 3, 457--465.

\bibitem{osserman} {\bf R. Osserman}, {\it A Survey of Minimal Surfaces}.
  Dover Publications, Inc., New York, 1986.

\bibitem{perez} {\bf J. Perez}, Limits by rescalings of minimal surfaces: minimal laminations, curvature decay and local pictures. Notes for the workshop ``Moduli Spaces of Properly Embedded Minimal Surfaces'', American Institute of Mathematics, Palo Alto, California (2005).

\bibitem{solomon-white} {\bf B. Solomon} and {\bf B. White}, A Strong
  Maximum Principle for Varifolds that are Stationary with Respect to Even
  Parametric Elliptic Functionals, Indiana J. Math {\bf 38} (1989) 683--691.

\bibitem{white1} {\bf B. White}, The space of minimal submanifolds for
varying Riemannian metrics.  Indiana Univ. Math. J. {\bf 40} (1991), no. 1,
161--200.

\bibitem{white2} {\bf B. White}, New applications of mapping degrees to
minimal surface theory.  J. Differential Geom. {\bf 29} (1989), no. 1,
143--162.

\bibitem{yang-yau} {\bf P. Yang} and {\bf S.T. Yau}.  Eigenvalues of the
Laplacian of compact Riemann surfaces and minimal submanifolds.
Ann. Scuola Norm. Sup. Pisa Cl. Sci. {\bf 7} (1980), no. 1, 55--63.

\end{thebibliography}
\end{document}